\documentclass[a4paper,11pt,twoside]{amsart}
\usepackage[latin1]{inputenc}
\usepackage[T1]{fontenc}

\usepackage{a4wide}

\usepackage{ifpdf}
\ifpdf\usepackage[pdftex]{hyperref}
\else\usepackage[hypertex]{hyperref}\fi

\usepackage{amsmath,amssymb,amsthm}

\theoremstyle{plain}
\newtheorem{thm}{Theorem}[section]
\newtheorem{prop}[thm]{Proposition}

\newtheorem{cor}[thm]{Corollary}

\theoremstyle{definition}

\theoremstyle{remark}
\newtheorem{rem}[thm]{Remark}

\theoremstyle{remark}

\DeclareMathOperator{\tr}{Tr}

\DeclareMathOperator{\im}{Im}

\let\dsp=\displaystyle 

\def\R{\mathbb R}

\def\N{\mathbb N}

\begin{document}

\title[Sobolev inequalities and spectral density]{Around a
  Sobolev--Orlicz inequality\\
  for operators of given spectral density}

\author{Michel Rumin}
\address{Laboratoire de Mathématiques d'Orsay\\
  CNRS et Université Paris Sud\\
  91405 Orsay Cedex\\ France}

\email{michel.rumin@math.u-psud.fr}

\date{\today}

\begin{abstract}
  We prove some general Sobolev--Orlicz, Nash and Faber--Krahn
  inequalities for positive operators $A$ of given ultracontractive
  spectral decay $F(\lambda)= \|\chi_A(]0,
  \lambda])\|_{1,\infty}$. For invariant operators on coverings of
  finite simplicial complexes this function is equivalent to von
  Neumann spectral density. This allows in the polynomial decay case
  to relate the Novikov--Shubin numbers to Sobolev inequalities on
  exact $\ell^2$-cochains, and to the vanishing of the torsion of the
  $\ell^{p,2}$-cohomology for some $p\geq 2$.
\end{abstract}

\keywords{Sobolev--Orlicz inequality, spectral distribution,
  Faber--Krahn inequality, Novikov--Shubin invariants,
  $\ell^{p,q}$-cohomology}

\subjclass[2000]{58J50, 46E35, 58J35, 46E30, 35P20.}

\thanks{Author supported in part by the French ANR-06-BLAN60154-01
  grant.}

\maketitle


\section{Introduction and main results}
\label{sec:introduction}

Let $A$ be a strictly positive self-adjoint operator on a measure
space $(X, \mu)$. Suppose moreover that the semigroup $e^{-tA}$ is
equicontinuous on $L^1(X)$. Then, according to Varopoulos
\cite{Varopoulos,Coulhon2}, a polynomial heat decay
\begin{displaymath}
  \|e^{-tA}\|_{1,\infty} \leq C t^{-\alpha/2} \quad \mathrm{with} \quad
  \alpha > 2 \,,
\end{displaymath}
is equivalent to the Sobolev inequality
\begin{equation}
  \label{eq:1}
  \|f \|_p \leq C'
  \|A^{1/2}f\|_2 \quad \mathrm{for} \quad
  1/p = 1/2 - 1/\alpha.  
\end{equation}
This result applies in particular in the case $A$ is the Laplacian
acting \emph{on scalar functions} of a complete manifold, either in
the smooth or discrete graph setting.

\smallskip

The first purpose of this paper is to present short proofs of general
Sobolev--Orlicz inequalities that hold for positive self-adjoint
operators, without equicontinuity or polynomial decay assumption,
knowing either their heat decay, as previously, or their
``ultracontractive spectral decay'' $F(\lambda) =
\|\Pi_\lambda\|_{1,\infty}$ of their spectral projectors $\Pi_\lambda
= \chi_A(]0, \lambda])$ on $E_\lambda$. As will be seen in
Sections~\ref{sec:ultr-norms-gamma} and \ref{sec:spectr-dens-cohom},
the interest for this former $F(\lambda)$ mostly comes from geometric
considerations. For instance if $A$ is a scalar invariant operator
over a discrete group $\Gamma$, or more generally an unimodular one,
then $F(\lambda) $ coincides with von Neumann's $\Gamma$-dimension of
$E_\lambda$, and thus $F$ represents the spectral density function of
$A$, see Proposition~\ref{prop:trace-norm}. In the general setting the
spectral decay $F$ stays a right continuous increasing function as
comes from the identity
\begin{equation}
  \label{eq:2}
  \|P^*P\|_{1, \infty} = \|P\|_{1, 2}^2 = \sup_{\|f\|_1,
    \|g\|_1 \leq 1}| \langle Pf, Pg \rangle |.
\end{equation}
We state the Sobolev--Orlicz inequalities we shall prove. In the
sequel, if $\varphi$ is a monotonic function, $\varphi^{-1}$ will
denote its right continuous inverse.
\begin{thm}
  \label{thm:1.1}
  Let $A$ be a positive self-adjoint operator on $(X, \mu)$ with
  ultracontractive spectral projections $\Pi_\lambda =
  \chi_A(]0,\lambda])$, i.e. $F(\lambda) = \|\Pi_\lambda\|_{1,\infty}
  < +\infty$.

  Suppose moreover that the Stieljes integral $\dsp G(\lambda) =
  \int_0^\lambda \frac{dF(u)}{u}$ converges.  Then any non zero $f \in
  L^2(X) \cap (\ker A)^\bot$ of finite energy $\mathcal{E}(f) =
  \langle Af, f\rangle_2$ satisfies
  \begin{equation}
    \label{eq:3}
    \int_X H \Bigl( \frac{|f(x)|^2}{4 \mathcal{E}(f)} \Bigr) d \mu \leq 1 \,, 
  \end{equation}
  where $H(y) = y \,G^{-1}(y)$.
\end{thm}
The heat version of this result has a similar statement (and proof).
\begin{thm}
  \label{thm:1.2}
  Let $A$ be a positive self-adjoint operator on $(X, \mu)$ such that
  $L(t) = \|e^{-tA} \Pi_V\|_{1,\infty}$ is finite, with $V= L^2(X)\cap
  (\ker A)^\bot$.
  
  Suppose moreover that $\dsp M(t) = \int_t^{+\infty} L(u) du < +
  \infty $. Then any non zero $f \in V$ of finite energy satisfies
  \begin{equation}
    \label{eq:4}
    \int_X N \Bigl( \frac{|f(x)|^2}{4 \mathcal{E}(f)} \Bigr) d \mu
    \leq \ln 2 \,,
  \end{equation}
  where $N(y) = y / M^{-1}(y)$
\end{thm}
Both results give (effective) Sobolev inequalities \eqref{eq:1} in the
polynomial decay case for $F$ or $L$. At first, one sees easily that
the transform from $F$ to $G$ is increasing, see \eqref{eq:13}, while
$G$ to $H$ is decreasing. Therefore, if $F(\lambda) \leq C
\lambda^\alpha$ for $\alpha > 1$, then $G(\lambda) \leq C_1
\lambda^{\alpha-1}$ with $C_1 = \frac{C \alpha }{\alpha-1}$, and $H(y)
\geq C_1^{\frac{1}{1-\alpha}} y^{\frac{\alpha}{\alpha-1}}$. Hence
(\ref{eq:3}) reads $\|f \|_{2\alpha/(\alpha-1)} \leq 2
C_1^{\frac{1}{2\alpha}} \|A^{1/2}f\|_2$.

\smallskip

Under convexity assumptions, $H$ and $N$-Sobolev inequalities
\eqref{eq:3} and \eqref{eq:4} imply some general Nash and Faber--Krahn
inequalities, see \eqref{eq:18} and~\eqref{eq:19}. This approach
assumes some thinness of the near-zero spectrum, as required by the
convergence of $G$ or $M$.  Since the classical Nash inequality makes
sense for thick spectrum, one may look for a direct proof. From heat
decay to Nash, such a derivation has already been obtained for general
operators by Coulhon, see \cite{Coulhon2} and the survey
\cite{Coulhon3}. Therefore we will focus here on the relationship
between the spectral density $F$ and Nash. This states as follows.
\begin{thm}
  \label{thm:1.3} Let $A$ be a positive self-adjoint operator, with
  finite $F (\lambda) = \|\Pi_\lambda\|_{1, \infty} $, and a non zero
  $f \in V= L^2(X)\cap (\ker A)^\bot$.
  
  $\bullet$ Then it holds that
  \begin{equation}
    \label{eq:5}
    \int_X |f(x)|^2 F^{-1}\Bigl( \frac{|f(x)|}{2 \|f\|_1}\Bigr) d \mu\leq 4
    \mathcal{E}(f)\,. 
  \end{equation}
  $\bullet$ If $\varphi$ is a convex function such that $0 \leq
  \varphi(y) \leq y F^{-1}(y)$, then the Nash--type inequality holds
  \begin{equation}
    \label{eq:6}
    \|f\|_1^2 \/ \varphi \Bigl(\frac{ \|f\|_2^2}{2 \|f\|_1^2}\Bigr) \leq 2
    \mathcal{E}(f) \,.
  \end{equation}
  $\bullet$ In particular if $f$ and $Af$ are supported in a domain
  $\Omega$ of finite measure, the Faber--Krahn type inequality is
  satisfied
  \begin{equation}
    \label{eq:7}
    \mu (\Omega) \varphi\Bigl(\frac{1}{2 \mu(\Omega)}\Bigr) \leq \frac{2
      \mathcal{E}(f)}{\|f\|_2^2} \,.
  \end{equation}
\end{thm}
As an illustration, if one can take above $\varphi (y) \geq C y
F^{-1}(y)$ for some constant $C$, for instance $\varphi (y)=
yF^{-1}(y) $ if convex itself, then \eqref{eq:7} shows that a non zero
state $f \in V$ of energy $\mathcal{E}(f) \leq \lambda \|f\|_2^2$ and
support $\Omega$ satisfies the simple uncertainty principle
\begin{equation}
  \label{eq:8}
  2 \mu(\Omega) F(4\lambda/C )  \geq 1\,.
\end{equation}
On groups this fits well with the interpretation of $F(\lambda)$ as a
renormalised ``density'' of dimension of $E_\lambda$ per volume. More
concretely, for any invariant positive scalar operator on a finite
group $\Gamma$, one has by Proposition~\ref{prop:trace-norm} that
$F(\lambda) = \dim E_\lambda / \mathrm{card} (\Gamma)$, and thus
\eqref{eq:8} reads
\begin{displaymath}
  2 \dim E_{4\lambda / C} \geq \mathrm{card} (\Gamma) /
  \mathrm{card} (\Omega) \,.
\end{displaymath}
Except for the multiplicative constants $2$ and $4/C$ this formula is
quite sharp in general. Indeed it could happen in some case that
$\Gamma$ be tiled by $N = \mathrm{card} (\Gamma) / \mathrm{card}
(\Omega)$ copies of such domains $\Omega$, implying by min-max
principle that $\dim E_\lambda \geq N$ there.

\medskip

The Sobolev-like inequalities in Theorem~\ref{thm:1.1} and
\ref{thm:1.2} are not restricted to scalar functions and apply in
particular to the following setting. Let $K$ be a finite simplicial
complex and $X \rightarrow K = X / \Gamma$ some covering. One
considers on $X$ the complex of $\ell^2$ $k$-cochains with the
discrete coboundary
\begin{displaymath}
  d_k : \ell^2 X^k \rightarrow \ell^2 X^{k+1} 
\end{displaymath}
dual to the usual boundary $\partial$ of simplexes, see e.g. \cite[\S
3]{Pansu1}.

Its ($\ell^2$) cohomology $H_2^{k+1} = \ker d_{k+1} / \im d_k$ splits
in two components :
\begin{itemize}
\item the reduced part $\overline H_2^{k+1} = \ker d_{k+1} /
  \overline{ \im d_k}$, isomorphic to $\ell^2$-harmonic cochains
  $\mathcal{H}_2^{k+1} = \ker d_{k+1} \cap \ker d_k^*$,
\item and the torsion $T_2^{k+1} = \overline{ \im d_k} / \im d_k$.
\end{itemize}
Although this torsion is not a normed space, one can study it by
``measuring'' the unboundedness of $d_k^{-1} $ on $\im d_k$.
We will consider here two different means.

\smallskip - A first one is inspired by $\ell^{p,q}$-cohomology. One
enlarges the space $\ell^2 X^k $ to $\ell^p X^k$ for $p \geq 2$, and
asks whether, for $p$ large enough, one has
\begin{equation}
  \label{eq:9}
  \overline{d_k
    (\ell^2 X^k)}^{\ell^2}  \subset d_k(\ell^p X^k  )\,, 
\end{equation} 
This is satisfied in case the following Sobolev identity holds
\begin{equation}
  \label{eq:10}
  \exists C \quad \mathrm{such\ that }\quad  \|\alpha\|_p \leq C \|d_k
  \alpha \|_2 \quad \mathrm{for\ all}\ \alpha 
  \in (\ker d_k)^\bot  \subset  \ell^2\,. 
\end{equation}
The geometric interest of the rougher formulation \eqref{eq:9} lies in
its stability under the change of $X$ into other bounded homotopy
equivalent spaces, as stated in
Proposition~\ref{prop:invariance-torsion}. Moreover if $\overline
H_2^{k+1}(X)$ vanishes, then \eqref{eq:9} is equivalent to the
vanishing of the torsion of the $\ell^{p,2}$-cohomology of $X$, as
will be seen in Section~\ref{sec:spectr-dens-cohom}.

\smallskip - The second approach is spectral and relies on von Neumann
$\Gamma$-dimension. Consider the $\Gamma$-invariant self-adjoint $A =
d_k^* d_k $ acting on $(\ker d_k)^\bot$ and the spectral density
$F_{\Gamma,k}(\lambda)= \dim_\Gamma E_\lambda$ of its spectral spaces
$E_\lambda$. This function vanishes near zero if and only if zero is
isolated in the spectrum of $A$, which is equivalent to the vanishing
of the torsion $T_2^{k+1}$. The asymptotic behaviour of
$F_{\Gamma,k}(\lambda)$ when $\lambda \searrow 0$ has a geometric
interest in general since, given $\Gamma$, it is an homotopy invariant
of the quotient space $K$, as shown by Efremov, Gromov and Shubin in
\cite{Efremov,Gromov-Shubin,Gromov-Shubin-erratum}.

One can compare these two notions in the spirit of Varopoulos result
\eqref{eq:1} on functions. In the case of polynomial decay one
obtains.
\begin{thm}
  \label{thm:1.4} Let $K$ be a finite simplicial space and $ X
  \rightarrow K = X / \Gamma$ a covering. Let $F_{\Gamma,k}(\lambda)=
  \dim_\Gamma E_\lambda $ denotes the spectral density function of $A
  = d_k^* d_k$ on $(\ker d_k)^\bot$.

  If $F_{\Gamma,k}(\lambda) \leq C \lambda^{\alpha/2}$ for some
  $\alpha > 2 $, then the Sobolev inequality \eqref{eq:10}, and the
  inclusion \eqref{eq:9}, hold for $1/p \leq 1/2 - 1/\alpha$.
\end{thm}

If moreover the reduced $\ell^2$-cohomology $\overline{H}^{k+1}_2(X)$
vanishes, this implies the vanishing of the $\ell^{p,2}$-torsion of
$X$, as stated in Corollary~\ref{cor:lp2}.

Other spectral decays than polynomial can be handled with
Theorem~\ref{thm:1.1}, leading then to a bounded inverse of $d_k$ from
$\im d_k \cap \ell^2$ into a more general Orlicz space given by $H$.

\section{Proofs of main inequalities}
\label{sec:proofs}

The first step towards Theorems~\ref{thm:1.1} and \ref{thm:1.2} is to
consider the ultracontrativity of the auxiliary operators $A^{-1}
\Pi_\lambda$ and $A^{-1}e^{-tA}\Pi_V$.
\begin{prop}
  \label{prop:2.1} $\bullet$ Let $A$, $F$ and $G$ be given as in
  Theorem~\ref{thm:1.1}. Then $A^{-1} \Pi_\lambda$ is ultracontractive
  with
  \begin{equation}
    \label{eq:11}
    \|A^{-1} \Pi_\lambda\|_{1, \infty}  \leq G(\lambda) =
    \int_0^\lambda \frac{dF(u)}{u} \,.
  \end{equation}
  $\bullet$ Let $A$, $L$ and $M$ be given as in Theorem~\ref{thm:1.2}.
  Then $A^{-1} e^{-tA}\Pi_V$ is ultracontractive with
  \begin{equation}
    \label{eq:12}
    \|A^{-1} e^{-tA} \Pi_V\|_{1,\infty}  \leq M(t) = \int_t^{+\infty}
    L(s) ds\,.
  \end{equation}
\end{prop}
\begin{proof}
  $\bullet$ The spectral calculus gives
  \begin{displaymath}
    A^{-1} (\Pi_\lambda - \Pi_\varepsilon)  = \int_{]\varepsilon,
      \lambda]} u^{-1} d \Pi_u  = \lambda^{-1} \Pi_\lambda -
    \varepsilon^{-1} \Pi_\varepsilon
    + \int_{]\varepsilon, \lambda]}
    u^{-2}\Pi_u du\,,   
  \end{displaymath}
  thus taking norms, one obtains
  \begin{align*}
    \|A^{-1} (\Pi_\lambda - \Pi_\varepsilon)\|_{1,\infty} & \leq
    \lambda^{-1} F(\lambda) + \varepsilon^{-1} F(\varepsilon) +
    \int_{]\varepsilon, \lambda]}
    u^{-2} F(u) du \\
    & = G(\lambda) -G (\varepsilon) + 2 \varepsilon^{-1}
    F(\varepsilon) \,.
  \end{align*}
  Now by finiteness of $G$, one has $\|\Pi_\varepsilon / \varepsilon
  \|_{1, \infty} = F(\varepsilon)/ \varepsilon \leq G(\varepsilon)
  \rightarrow 0$ when $\varepsilon \searrow 0$, hence by \eqref{eq:2}
  \begin{align*}
    \|A^{-1}\Pi_\lambda\|_{1,\infty}
    & = \|\Pi_\lambda A^{-1/2}\Pi_\lambda\|_{1,2}^2 \\
    & = \lim_{\varepsilon \rightarrow 0} \| (\Pi_\lambda -
    \Pi_\varepsilon)A^{-1/2} \Pi_\lambda\|_{1,2}^2 \quad \textrm{by
      Beppo-Levi,}\\
    & = \lim_{\varepsilon \rightarrow 0} \|A^{-1} (\Pi_\lambda -
    \Pi_\varepsilon)\|_{1,\infty} \leq G(\lambda)\,.
  \end{align*}
  We note that we also have
  \begin{equation}
    \label{eq:13}
    G(\lambda) = \lambda^{-1} F(\lambda) + \int_0^\lambda u^{-2} F(u) du\,,
  \end{equation}
  which shows the useful monotonicity of the transform from $F$ to $G$
  and $H$.

  $\bullet$ The heat case \eqref{eq:12} is clear since $A^{-1} e^{-tA}
  \Pi_V = \int_t^{+\infty} e^{-sA}\Pi_V ds$ by the spectral calculus.
\end{proof}
 
The sequel of the proofs of Theorems~\ref{thm:1.1},~\ref{thm:1.2} and
~\ref{thm:1.3} relies on a classical technique from real interpolation
theory, as used for instance in the elementary proof of the $L^2-L^p$
Sobolev inequality in $\R^n$ given by Chemin and Xu in
\cite{Chemin-Xu}. This consists here in estimating a level set $\{x,
|f(x)| > y\}$ by using an appropriate spectral splitting of $f =
\Pi_\lambda f + \Pi_{> \lambda}f $ for $f\in V$.

\subsection{Proof of Theorem~\ref{thm:1.1}}
\label{sec:proof-H-Sobolev}
By (\ref{eq:2}) and (\ref{eq:11}) one has $\|A^{-1/2} \Pi_\lambda
\|_{2, \infty}^2 \leq G(\lambda)$, hence
\begin{equation}
  \label{eq:14}
  \|\Pi_\lambda f\|_\infty^2 \leq G(\lambda) \|A^{1/2} f\|^2_2 =
  G(\lambda) \mathcal{E}(f)\,. 
\end{equation}
Then suppose that $|f(x)| \geq y$, with $y^2 = 4 G(\lambda)
\mathcal{E}(f)$. As $|\Pi_\lambda f(x)| \leq y/2$ by \eqref{eq:14},
one has necessarily $|\Pi_{> \lambda} f(x)| \geq y /2 \geq
|\Pi_\lambda f(x)|$ and finally
\begin{equation}
  \label{eq:15}
  |f(x) |^2 \leq 4 |\Pi_{> \lambda} f(x)|^2 \quad \mathrm{on} \quad
  \bigl\{ x \in X \mid |f(x)|^2 \geq 4 G(\lambda)
  \mathcal{E}(f) \bigr\} \,. 
\end{equation}
Hence a first integration in $x$ gives,
\begin{displaymath}
  \int_{\{x \, , \, |f(x)|^2 \geq 4 \mathcal{E}(f) G(\lambda)\}}
  |f(x)|^2 d\mu \leq 4 \|\Pi_{> \lambda} f\|_2^2 \,,
\end{displaymath}
and a second integration in $\lambda$,
\begin{displaymath}
  \int_X \frac{|f(x)|^2}{4 \mathcal{E}(f)} G^{-1}\Bigl( \frac{|f(x)|^2}{4
    \mathcal{E}(f)} \Bigr) d \mu (x) \leq \int_0^{+\infty} \frac{\|\Pi_{> \lambda}
    f\|_2^2 }{  \mathcal{E}(f)} d \lambda \,, 
\end{displaymath}
where $G^{-1}(y) = \sup\{\lambda \mid G(\lambda) \leq y \}$. At last
the spectral calculus provides
\begin{align*}
  \int_0^{+\infty} \|\Pi_{> \lambda} f\|_2^2 \, d \lambda & =
  \int_0^{+\infty } \int_\lambda^{+\infty}
  \langle d \Pi_\mu f, f\rangle  \\
  & = \int_0^{+\infty} \mu \, \langle d \Pi_\mu f, f\rangle = \langle
  Af, f \rangle = \mathcal{E}(f) \,,
\end{align*}
giving Theorem~\ref{thm:1.1}.

\subsection{Proof of Theorem~\ref{thm:1.2}}
\label{sec:proof-N-Sobolev}
We follow the same lines as above. First by \eqref{eq:2} and
\eqref{eq:12} one has for $f\in V$
\begin{displaymath}
  \|e^{-tA/2} f\|_\infty \leq M(t) \mathcal{E}(f)\,,
\end{displaymath}
leading to
\begin{equation}
  \label{eq:16}
  |f(x) |^2 \leq 4 |(1- e^{-tA/2}) f(x)|^2 \quad \mathrm{on} \quad
  \bigl\{ x \in X \mid |f(x)|^2 \geq 4 M(t) \mathcal{E}(f)
  \bigr\}\,.
\end{equation}
Then integrations in $x$ and $dt/t^2$ give
\begin{displaymath}
  \int_X \frac{|f(x)|^2}{4 \mathcal{E}(f)} / M^{-1}\Bigl( \frac{|f(x)|^2}{4
    \mathcal{E}(f)} \Bigr) \, d \mu (x) \leq \frac{1}{\mathcal{E}(f)}
  \int_0^{+\infty} \|(1- 
  e^{-tA/2}) f\|_2^2 \, \frac{dt}{t^2} \,, 
\end{displaymath}
where now $M^{-1}(y)= \inf \{t \mid M(t) \geq y\}$ for the decreasing
$M$. The right integral is computed by spectral calculus
\begin{align*}
  \int_0^{+\infty} \|(1- e^{-tA/2}) f\|_2^2 \, \frac{dt}{t^2} & =
  \int_0^{+\infty} \int_0^{+\infty} (1- e^{-t\lambda/2})^2 \langle d
  \Pi_\lambda f, f
  \rangle \, \frac{dt}{t^2}  \\
  & = \int_0^{+\infty} \Bigl( \int_0^{+\infty} \frac{(1- e^{-u})^2 }{2
    u^2} du \Bigr) \lambda \langle d \Pi_\lambda f, f \rangle \\
  & = I \mathcal{E}(f) \,,
\end{align*}
where $\dsp 2 I = \int_0^{+\infty} \frac{(1- e^{-u})^2}{u^2}du = 2 \ln
2$ as seen developing $\dsp I_\varepsilon = \int_\varepsilon^{+\infty}
\frac{(1- e^{-u})^2}{u^2}du$ when $\varepsilon \searrow 0$.

\subsection{Proof of Theorem~\ref{thm:1.3}}
\label{sec:proof-Nash}

Here one compares levels of $f$ to $\|f\|_1$ instead of
$\mathcal{E}(f)$. Using $F(\lambda) = \|\Pi_\lambda\|_{1,\infty}$ one
gets
\begin{equation}
  \label{eq:17}
  |f(x) |^2 \leq 4 |\Pi_{> \lambda} f(x)|^2 \quad \mathrm{on} \quad
  \bigl\{ x \in X \mid |f(x)| \geq 2 F(\lambda) \|f\|_1 \bigr\}\,. 
\end{equation}
This leads to \eqref{eq:5} by integration as before, from which
follows the Nash--type inequality \eqref{eq:6} by applying Jensen
inequality to the convex function $\varphi$ and the probability
measure $d P = |f| \, d \mu/\|f\|_1$.

\begin{rem}
  In the previous proofs, it appears clearly that the proposed
  controls of ultracontractive norms of spectral or heat decay are
  much stronger than the Sobolev and Nash-type inequalities
  deduced. Indeed these inequalities are twice integrated versions, in
  space and frequency, of the ``local'' inequalities \eqref{eq:15},
  \eqref{eq:16} and \eqref{eq:17}, that come directly from the
  ultracontractive controls.  Therefore it seems hopeless to get the
  converse statements in general. However one can get back from
  Sobolev or Nash to heat decay, in the case the heat is
  equicontinuous on $L^1$; as due to Varopoulos in \cite{Varopoulos}
  for the polynomial case, and Coulhon in \cite{Coulhon2} for more
  general decays.
\end{rem}

\section{Relationships between inequalities}
\label{sec:relat-betw-ineq}

\subsection{From H-Sobolev to Nash}
\label{sec:h-sobolev-nash}
We compare and comment briefly the various results obtained.  At
first, in the classical polynomial case, Sobolev inequality
(\ref{eq:1}) implies Nash' one
$$
\|f\|_2^{1+ 2/\alpha} \leq C \|f\|_1^{2/\alpha} \mathcal{E}(f)^{1/2}
$$
by H\"older, see e.g. \cite{Coulhon3}. In the general case here one
needs some convexity assumptions to get a Nash--type inequality from H
or N-Sobolev.

Indeed, suppose either the H or N-Sobolev inequality \eqref{eq:3} or
\eqref{eq:4} holds, and suppose $\varphi$ is a convex function such
that $\varphi(y) \leq y G^{-1}(y^2)$, resp. $\varphi (y) \leq y
/M^{-1}(y^2)$. Then by Jensen the following Nash-type inequality is
satisfied
\begin{equation}
  \label{eq:18}
  \varphi \Bigl( \frac{\|f\|_2^2}{2 \mathcal{E}(f)^{1/2} \|f\|_1}
  \Bigr)  \leq \frac{2\mathcal{E}(f)^{1/2} }{\|f\|_1}\ \mathrm{resp.\ }
  \frac{2\ln 2\mathcal{E}(f)^{1/2} }{\|f\|_1} \,.
\end{equation}
If one can take $\varphi (y) \geq C y G^{-1}(y^2)$ for some constant
$C$, this leads to
\begin{equation}
  \label{eq:19}
  \frac{\|f\|_2^2}{\|f\|_1^2} \leq
  \frac{4\mathcal{E}(f)}{\|f\|_2^2} G\Bigl(
  \frac{4\mathcal{E}(f)}{C\|f\|_2^2} \Bigr)  
\end{equation}
In comparison, the Nash inequality \eqref{eq:6} provides
\begin{equation}
  \label{eq:20}
  \frac{\|f\|_2^2}{2 \|f\|_1^2} \leq
  F\Bigl(
  \frac{4\mathcal{E}(f)}{C\|f\|_2^2} \Bigr) \,, 
\end{equation}
if there exists a convex function $\psi$ such that $C y F^{-1}(y)\leq
\psi(y) \leq y F^{-1}(y)$.  Up to constants this latter formula
\eqref{eq:20} is a priori sharper than \eqref{eq:19}, since
$F(\lambda) \leq \lambda G(\lambda)$ in general.

Observe that one may have $F (\lambda) \ll \lambda G(\lambda)$ for
very thick near-zero spectrum. For instance if $F(\lambda) = \lambda /
\ln^2 \lambda$ then $\lambda G(\lambda)= (-\ln \lambda + 1
)F(\lambda)$. Except this ``low dimensional'' phenomenon, one has
$\lambda G(\lambda) \underset{0}{\asymp} F(\lambda)$ in the other
cases, and thus the two Nash inequalities \eqref{eq:19} and
\eqref{eq:20} have same strength. For instance this holds if
$F(\lambda) \underset{0}{\sim} \lambda^{1+\varepsilon}
\varphi(\lambda)$ for some $\varepsilon > 0$ and an increasing
$\varphi>0$. This comes from the following remark.
\begin{prop}
  \label{prop:3:1}Suppose there exists $\varepsilon > 0$ such that,
  for small $\lambda$, $F$ satisfies the growing condition $
  F(2\lambda) \geq 2 (1+ \varepsilon) F(\lambda)$, then
  $(2+\varepsilon^{-1}) F(\lambda) \geq \lambda G(\lambda) \geq
  F(\lambda)$.
\end{prop}
\begin{proof}
  By \eqref{eq:13}, one has
  \begin{align*}
    G(\lambda) & = \int_0^\lambda \frac{dF(u)}{u} =
    \frac{F(\lambda)}{\lambda} + \int_0^\lambda \frac{F(u)}{u^2} du \\
    & = \frac{F(\lambda)}{\lambda} + \Bigl( \int_0^{\lambda/2} +
    \int_{\lambda/2}^\lambda \Bigr) \frac{F(u)}{u^2} du \\
    & \leq \frac{2 F(\lambda)}{\lambda} + \int_0^{\lambda/2}
    \frac{F(2u)}{2 (1+\varepsilon) u^2} du \quad \mathrm{ by\
      hypothesis\ on\ }F\,,\\
    & \leq \frac{2 F(\lambda)}{\lambda} + \frac{1}{1+\varepsilon}
    \Bigl( G(\lambda) - \frac{ F(\lambda)}{\lambda} \Bigr)\,,
  \end{align*}
  leading to $\lambda G(\lambda) \leq (2 + \varepsilon^{-1})
  F(\lambda)$\,.
\end{proof}
As a curiosity, we note that under the growing hypothesis on $F$
above, the spectral density of states $F$ and the spatial repartition
function $H$ have symmetric expressions with respect to $G$ and
$G^{-1}$. Indeed, one has simply there
\begin{equation}
  \label{eq:21}
  F(\lambda) \asymp \lambda G(\lambda) \quad \mathrm{while } \quad
  H(x) = x G^{-1}(x)\,.
\end{equation}

  

\subsection{Spectral versus heat decay}
\label{sec:spectral-heat-decay}

One would like to compare the two Theorems~\ref{thm:1.1}
and~\ref{thm:1.2}. They both lead to Sobolev inequalities starting
either from the heat or spectral decay. One can compare $F$ and $G$ to
$L$ and $M$ through Laplace transform of associated measures.
\begin{prop}
  \label{prop:3:2}
  $\bullet $ In any case it holds that
  \begin{gather}
    \label{eq:22}
    L(t) \leq \mathcal{L}(dF)(t) = \int_0^{+\infty} e^{-\lambda
      t} d F(\lambda) \\
    \label{eq:23}
    M(t) \leq \mathcal{L}(dG)(t) = \int_0^{+\infty} e^{-\lambda t} d
    G(\lambda)\,.
  \end{gather}
  
  $\bullet$ If $A$ is an invariant operator acting on $L^2$-sections
  of an invariant vector bundle $V$ over a locally compact group
  $\Gamma$, then reverse inequalities hold up to the multiplicative
  factor $n = \dim V$, i.e.
  \begin{displaymath}
    \mathcal{L}(dF) \leq n L \quad \mathrm{and} \quad \mathcal{L}(dG)
    \leq n M\,.
  \end{displaymath}
  Moreover $G(y) \leq n e M(y^{-1})$ and H-Sobolev inequality
  \eqref{eq:3} implies N-Sobolev \eqref{eq:4}, up to multiplicative
  constants.

  $\bullet$ Reversely, for any operator, if $G$ satisfies the
  exponential growing condition :
  \begin{displaymath}
    \exists C \ \mathrm{such\ that\
    }\forall u, y > 0\,,\  G(uy) \leq
    e^{C u} G(y) \,,
  \end{displaymath}
  then $ M(y^{-1}) \leq 3 G(2Cy)$. Hence $H$ and $N$-Sobolev are
  equivalent on groups in that case.
\end{prop}
\begin{proof}
  $\bullet $ By spectral calculus $e^{-tA}\Pi_V = \int_0^{+\infty}
  e^{-t\lambda} d \Pi_\lambda = t \int_0^{+\infty} e^{-t \lambda}
  \Pi_\lambda d \lambda$, hence
  \begin{displaymath}
    L(t) = \| e^{-t A} \Pi_V\|_{1, \infty} \leq t \int_0^{+\infty} e^{-t
      \lambda} \|\Pi_\lambda\|_{1, \infty} d\lambda =
    \mathcal{L}(dF)(t) \,,
  \end{displaymath}
  and thus
  \begin{displaymath}
    M(t) = \int_t^{+\infty} L(s) ds  \leq
    \int_t^{+\infty} \int_0^{+\infty} e^{-\lambda s } d F(\lambda)
    ds = \int_0^{+\infty} \frac{e^{-\lambda t}}{\lambda} dF(\lambda) =
    \mathcal{L}(dG) (t)\,.  
  \end{displaymath}
  
  $\bullet$ For positive invariant operators $P$ on groups, we will
  see in Proposition~\ref{prop:trace-norm} that the ultracontractive
  norm $\|P\|_{1, \infty}$ is pinched between the trace $\tau_\Gamma
  (P)$ and $n \tau_\Gamma(P)$. This gives the reverse inequalities by
  linearity of $\tau_\Gamma$. In particular one gets
  \begin{align*}
    n M( y^{-1}) & \geq \int_0^{+\infty} e^{-\lambda/y} dG (\lambda) =
    y^{-1}
    \int_0^{+\infty} e^{-\lambda/y} G(\lambda) d \lambda \\
    & \geq y^{-1} \int_y^{+\infty} e^{-\lambda/y} G(y) d\lambda =
    e^{-1} G(y).
  \end{align*}
  Therefore $N(y) = y / M^{-1}(y^{-1}) \leq y G^{-1}(ey) = e^{-1}
  H(ey)$ and $H$-Sobolev implies
  \begin{displaymath}
    \int_X N \Bigl( \frac{|f(x)|^2}{4e \mathcal{E}(f)} \Bigr) d \mu
    \leq e^{-1}\,. 
  \end{displaymath}
  
  $\bullet$
  If $G$ satisfies the growing condition, one has by (\ref{eq:23})
  \begin{align*}
    M(1/y) & \leq \int_0^{+\infty} e^{-\lambda/y} d G(\lambda) =
    \int_0^{+\infty} e^{-u} G(u y) d u \\
    & \leq \int_0^{2C} e^{-u} G(2C y) du + \int_{2C}^{+\infty} e^{-
      u/2}
    G(2C y) du  \\
    & \leq 3 G(2Cy) \,.
  \end{align*}
\end{proof}

We note that it may happen that $N \ll H$ for very thin near--zero
spectrum. In an extreme case there may be a gap in the spectrum,
i.e. $A \geq \lambda_0 > 0$, hence $F = G = H = 0$ near zero, while
$L(t) \asymp C e^{-c t} $, $M (t) \asymp C' e^{-c t}$ and $N(y) \asymp
C'' y / \ln (y/C')$.

\section{Ultracontractive norms and $\Gamma$-trace.}
\label{sec:ultr-norms-gamma}

For applications we now discuss some geometric aspect of the analytic
spectral decay $F(\lambda) = \|\Pi_\lambda\|_{1,\infty}$ we consider.

In the case of operators invariant under the action of a group
$\Gamma$, such hypercontractive norms are related to von Neumann
$\Gamma$-dimension and trace. We briefly recall these notions and
refer for instance to \cite[\S 2]{Pansu1} for more details. However we
will follow here a slightly different approach, as in \cite[\S
6.1]{Rumin05} for instance, that covers also some non-discrete
actions.

Suppose that a locally compact group $\Gamma$ (discrete or not) acts
by measure preserving transforms on the space $X$ with a \emph{finite}
quotient $X/\Gamma$. For instance, when $\Gamma$ is discrete, $X$ may
be a covering space over a finite simplicial complex. Equivalently one
can also take a $d$--dimensional invariant bundle $V$ over a group
$\Gamma$ and set $X = \Gamma \times [1,d]$, so that $L^2(X) \simeq
L^2(\Gamma) \otimes V_e$.

The following straightforward proposition, see
e.g. \cite[Prop. 6.4--6.6]{Rumin05}, leads to a definition of a
``$\Gamma$-trace'' in this setting.
\begin{prop}
  \label{prop:traceclass}
  Let $\Gamma$ be a locally compact group and $P$ be a
  $\Gamma$-invariant positive operator on $L^2(\Gamma) \otimes
  V_e$. For any $D\subset \Gamma$ with Haar measure $0 < \lambda(D) <
  +\infty$, consider the trace
  \begin{displaymath}
    \tau_D(P) = \lambda(D)^{-1}\tr(\chi_D P \chi_D)\,.
  \end{displaymath}
  $\bullet$ Let $S$ be the positive square root of $P$. Then $\tau_D
  (P)$ is finite iff $S\chi_D$ is an Hilbert--Schmidt operator. In
  that case the kernel of $S$ is $K_S(x,y) = k_S(y^{-1}x)$ with $k_S
  \in L^2(\Gamma)$, while the kernel of $P$ is $K_P (x,y ) =
  k_P(y^{-1} x)$ with $k_P = k_S * k_S \in C_0(\Gamma)$, and it holds
  that
  \begin{displaymath}
    \tau_D(P) = \int_\Gamma \tr_{V_e}\bigl(k_S^*(x) k_S(x) \bigr) d\lambda(x)  =
    \tr_{V_e} (k_P(e))\,.  
  \end{displaymath}
  In particular this trace is independent of $D$. It will be denoted
  by $\tau_\Gamma$ and called (improperly) the $\Gamma$--trace in the
  sequel.
  
  $\bullet$ If moreover $\Gamma$ is \emph{unimodular}, and $P$ is a
  (not necessarily positive) $\Gamma$--invariant bounded operator,
  then $\tau_\Gamma (P^*P) = \tau_\Gamma(P P^*)$. Hence $\tau_\Gamma$
  actually defines a faithful trace in that case.
\end{prop}

We recall that this last trace property allows to get a meaningful
notion of dimension for closed $\Gamma$-invariant subspaces $L \subset
H = L^2(\Gamma) \otimes V_e$. Indeed, one sets then $\dim_\Gamma L =
\tr_\Gamma(\Pi_L) $. This satisfies the key property $\dim_\Gamma f(L)
= \dim_\Gamma L$ for any closed densely defined invariant injective
operator $f: L \rightarrow H$, see e.g. \cite[\S 2]{Pansu1} or
\cite[\S 3.2]{Rumin05}.

On any locally compact group, the $\Gamma$-trace of $P$ is easily
compared to its ultracontractive norm.
\begin{prop}
  \label{prop:trace-norm}
  Let $P$ be a positive $\Gamma$-invariant operator acting on $L^2(X)
  = L^2(\Gamma)\otimes V_e$ with kernel $K_P (x,y) = k_P(y^{-1}x)$,
  then
  \begin{displaymath}
    \|P\|_{1, \infty} = \|k_P(e)\| \leq \tau_\Gamma(P) \leq (\dim V_e)
    \|P\|_{1, \infty}\,.  
  \end{displaymath}
\end{prop}

\begin{proof}
  In general one has $\dsp \|P\|_{1, \infty} = \sup_{x,y}
  \|K_P(x,y)\|$, and by positivity of $P$,
  \begin{displaymath}
    2 |\langle K_P(x,y)u,
    v\rangle | \leq  \langle K_P(x,x) u, u \rangle + \langle
    K_P(y,y) v, v \rangle\,.
  \end{displaymath}
  Therefore $\dsp \|P\|_{1, \infty} = \sup_x\|K_P(x,x)\| = \|k_P(e)\|$
  for an invariant operator. Here
  $$
  \|k_P(e)\| = \sup_{\|v\|\leq 1 } \|k_P(e) v\|_{V_e} =
  \sup_{\|v\|\leq 1} \langle k_P(e) v,v \rangle
  $$ 
  for the positive $k_P(e)$, while $\tau_\Gamma(P) =
  \tr_{V_e}(k_P(e))$ by Proposition \ref{prop:traceclass}.
\end{proof}

As a consequence, already used in Proposition~\ref{prop:3:2}, the norm
$\|P\|_{1,\infty}$ is, up to multiplicative constants, a linear form
on positive $P$. This gives also the converse inequalities to
\eqref{eq:11} and \eqref{eq:12} in Proposition~\ref{prop:2.1} for
invariant operators on groups. Indeed it holds in this case that
\begin{equation}
  \label{eq:24}
  \begin{aligned}
    \|A^{-1} \Pi_\lambda\|_{1, \infty} & \asymp \tau_\Gamma(A^{-1}
    \Pi_\lambda) \asymp G(\lambda)\\
    \|A ^{-1} e^{-tA} \|_{1, \infty} & \asymp \tau_\Gamma (A^{-1}
    e^{-tA}) \asymp M(t)\,,
  \end{aligned}
\end{equation}
due to the equalities $\tau_\Gamma(A^{-1} e^{-tA}) = \int_t^{+\infty}
\tau_\Gamma(e^{-s A}) ds $ and
\begin{displaymath}
  \tau_\Gamma(A^{-1} \Pi_\lambda) =
  \int_0^\lambda u^{-1} d \tau_\Gamma(\Pi_u) = \lambda^{-1}
  \tau_\Gamma(\Pi_\Lambda) + \int_0^\lambda u^{-2} \tau_\Gamma(\Pi_u)
  du \,. 
\end{displaymath}

\medskip

Its relation to the $\Gamma$-trace allows to estimate the
ultracontractive spectral decay $F(\lambda)$ of $A$ in some simple
cases. Namely, following Dixmier \cite[\S 18.8]{Dixmier}, if the group
$\Gamma$ is locally compact unimodular and postliminaire, there exists
a Plancherel measure $\mu$ on its unitary dual $\widehat \Gamma$,
together with a Plancherel formula that gives here
\begin{equation}
  \label{eq:25}
  F(\lambda) = \|\Pi_\lambda\|_{1,\infty} \asymp
  \tau_\Gamma(\Pi_\lambda) = \int_{\widehat G} 
  \tr(\widehat{\Pi_\lambda}(\xi)) d \mu(\xi)\,.  
\end{equation}

For instance, in the case of the Laplacian $\Delta$ on $\R^n$, the
spectral space $E_\lambda(\Delta)$ is the Fourier transform of
functions supported in the ball $B(0, \sqrt\lambda)$ in
$(\widehat{\R^n}, d\mu) \simeq (\R^n, (2\pi)^{-n} dx)$, hence
\begin{displaymath}
  F(\lambda) = \mu(B(0, \sqrt\lambda)) = C_n \lambda^{n/2} ,
\end{displaymath}
with $C_n= (2\pi)^{-n} \mathrm{vol}(B_n)$. This leads to
\begin{displaymath}
  G(\lambda) =
  \frac{n C_n}{n-2} \lambda^{n/2 -1}\quad \mathrm{and}\quad H(x) = x
  G^{-1}(x) = \Bigl(\frac{n-2}{n C_n}\Bigr)^{\frac{2}{n-2}} x^{\frac{n}{n-2}} \,,
\end{displaymath}
so that finally (\ref{eq:3}) gives the classical Sobolev inequality in
$\R^n$
\begin{displaymath}
  \|f\|_{2n/(n-2)} \leq \frac{1}{\pi} \Bigl(
  \frac{n\,\mathrm{vol}(B_n)}{n-2} \Bigr)^{\frac{1}{n}} \|df\|_2 \,. 
\end{displaymath}
Yet we recall that the best constant here is $2(n(n-2))^{-1/2}
\mathrm{area}(S^n)^{-1/n}$, see \cite{Aubin}.

Still on $\R^n$, one can get some general algebraic expression of
$F(\lambda)$ for positive invariant differential operator $A = \sum_I
a_I \partial_{x_I}$. Let $\sigma(A)(\xi) = \sum_I a_I (i\xi)^I$ be its
polynomial symbol. Then again the spectral space $E_\lambda(A)$
consists in functions whose Fourier transform is supported in
\begin{displaymath}
  D_\lambda= \{\xi \in \R^n \mid \sigma(A)(\xi) \leq \lambda\}
\end{displaymath}
and
\begin{displaymath}
  F(\lambda) = (2\pi)^{-n} \mathrm{vol}(D_\lambda). 
\end{displaymath}
The asymptotic behaviour of $F(\lambda)$ when $\lambda \searrow 0$ can
be obtained from the resolution of the singularity of the polynomial
$\sigma(A)$ at $0$. Indeed, there exists $\alpha \in \mathbb{Q}^+$ and
$k \in [0,n-1]\cap \N$ such that
\begin{displaymath}
  F(\lambda) \underset{\lambda \rightarrow 0^+}{\sim} C \lambda^\alpha |\ln
  \lambda|^k  \,,
\end{displaymath}
see e.g. Theorem 7 in \cite[\S21.6]{Arnold}. Moreover, under a
non-degeneracy hypothesis on $\sigma(A)$, the exponents $\alpha$ and
$k$ can be read from its Newton polyhedra. Then if $\alpha > 1$,
Proposition~\ref{prop:3:1} yields that $G(\lambda)
\underset{0}{\asymp} \lambda^{\alpha-1} |\ln \lambda|^k$. Therefore
$G^{-1}(u) \underset{0}{\asymp} u^{1/(\alpha-1)} |\ln
u|^{-k/(\alpha-1)}$ and finally the $H$-Sobolev inequality
\eqref{eq:3} is governed in small energy by the function
\begin{displaymath}
  H(u) \asymp u^{\frac{\alpha}{\alpha-1}}
  |\ln(u)|^{-\frac{k}{\alpha-1}} \quad \mathrm{for }\quad u \ll 1\,. 
\end{displaymath}

\section{Spectral density and cohomology}
\label{sec:spectr-dens-cohom}

To apply the previous results, we suppose now that $K$ is a finite
simplicial complex and consider a covering $\Gamma \rightarrow X
\rightarrow K$. Let $d_k$ be the coboundary operator on $k$-cochains
$X^k$ of $X$. As a purely combinatorial and local operator, it acts
boundedly on all $\ell^p$-spaces of cochains $\ell^p X^k$, see
e.g. \cite{Bourdon-Pajot,Pansu1}.

Let $F_{\Gamma,k}(\lambda)$ denotes the $\Gamma$-trace of the spectral
projector $\Pi_\lambda = \chi(]0, \lambda])$ of $A= d_k^* d_k$. By
Proposition~\ref{prop:trace-norm} this function is equivalent, up to
multiplicative constants, to the hypercontractive spectral decay
$F(\lambda) = \|\Pi_\lambda\|_{1, \infty}$. Thus Theorem~\ref{thm:1.4}
is a direct application of Theorem~\ref{thm:1.1} in the polynomial
case. This statement compares two measurements of the torsion of
$\ell^2$-cohomology $T_2^{k+1} = \overline{ d_k(\ell^2)}^{\ell^2} /
d_k(\ell^2)$ that share some geometric invariance. We describe this
more precisely.

\medskip

We first recall the main invariance property of
$F_{\Gamma,k}(\lambda)$. We say that two increasing functions $f, g:
\R^+ \rightarrow \R^+$ are equivalent if there exists $C \geq 1$ such
that $f(\lambda/C) \leq g(\lambda) \leq f(C \lambda)$ for $\lambda$
small enough. According to
\cite{Efremov,Gromov-Shubin,Gromov-Shubin-erratum} we have :
\begin{thm}
  Let $K$ be a finite simplicial complex and $ \Gamma \rightarrow X
  \rightarrow K $ a covering. Then the equivalence class of
  $F_{\Gamma,k}$ only depends on $\Gamma$ and the homotopy class of
  the $(k+1)$-skeleton of $K$.
\end{thm}
One tool in the proof is the observation that an homotopy of finite
simplicial complexes $F$ and $G$ induces bounded $\Gamma$-invariant
homotopies between the Hilbert complexes $( \ell^2 X^k, d_k)$ and
$(\ell^2 Y^k, d'_k)$. That means there exist $\Gamma$-invariant
bounded maps
\begin{displaymath}
  f_k : \ell^2 X^k \rightarrow \ell^2 Y^k \quad
  \mathrm{and} \quad g_k : \ell^2  Y^k
  \rightarrow \ell^2  X^k
\end{displaymath}
such that
\begin{displaymath}
  f_{k+1} d_k = d'_kf_k \quad \mathrm{and}\quad g_{k+1} d'_k =
  d_k g_k
\end{displaymath}
and
\begin{displaymath}
  g_k f_k = \mathrm{Id} + d_{k-1}h_k + h_{k+1} d_k \quad \mathrm{and}
  \quad
  f_k g_k = \mathrm{Id} + d'_{k-1} h'_k + h'_{k+1} d'_k
\end{displaymath} 
for some bounded maps
\begin{displaymath}
  h_k : \ell^2  X^k \rightarrow \ell^2
  X^{k-1} \quad \mathrm{and} \quad h'_k : \ell^2 
  Y^k \rightarrow 
  \ell^2  Y^{k-1}. 
\end{displaymath}
All these maps are purely combinatorial and local, see
e.g. \cite{Bourdon-Pajot,Rumin05}, and thus extend on all $\ell^p$
spaces of cochains.

One can show a similar invariance property of the inclusion
\eqref{eq:9} we recall below, but that holds more generally on
uniformly \emph{locally finite simplicial complexes}, without
requiring a group invariance. These are simplicial complexes such that
each point lies in a bounded number $N(k)$ of $k$-simplexes.

\begin{prop}
  \label{prop:invariance-torsion}
  Let $X$ and $Y$ be uniformly locally finite simplicial
  complexes. Suppose that they are boundedly homotopic in $\ell^2$ and
  $\ell^p$ norms for some $p \geq 2$. Then one has
  \begin{displaymath}
    \overline{d_k (\ell^2 X^k)}^{\ell^2} \subset
    d_k(\ell^p X^k )\,, \leqno{(9)}
  \end{displaymath}
  if and only if a similar inclusion holds on $Y$.
\end{prop}
\begin{proof}
  Suppose that $\overline{d_k (\ell^2 X^k)}^{\ell^2} \subset
  d_k(\ell^p X^k )$ and consider a sequence $\alpha_n = d'_k (\beta_n)
  \in d'_k(\ell^2 Y^k)$ that converges to $\alpha \in \overline{d_k
    (\ell^2 Y^k)}^{\ell^2}$ in $\ell^2$.

  Then $g_{k+1} \alpha_n = d_k(g_k \beta_n) \rightarrow g_{k+1} \alpha
  \in \overline{d_k (\ell^2 X^k)}^{\ell^2}$. Therefore there exists
  $\beta \in \ell^p X^k$ such that $g_{k+1} \alpha = d_k \beta $. Then
  taking $\ell^2$-limit in the sequence
  \begin{displaymath}
    f_{k+1} g_{k+1} \alpha_n = \alpha_n + d'_k h'_{k+1} \alpha_n + h'_{k+2}
    d'_{k+1} \alpha_n  = \alpha_n + d'_k h'_{k+1} \alpha_n
  \end{displaymath}
  gives
  \begin{displaymath}
    d'_k (f_k \beta) = f_{k+1} d_k \beta = \alpha + d'_k
    h'_{k+1}\alpha \,,
  \end{displaymath}
  and finally $\alpha \in d'_k(\ell^p Y^k)$ since $\ell^2 Y^k \subset
  \ell^p Y^k$ for $p \geq 2$.
\end{proof}

The inclusion \eqref{eq:9} we consider here is related to problems
studied in $\ell^{p,q}$ cohomology. We briefly recall this notion and
refer for instance to \cite{Goldshtein-Troyanov} for details. If $X$
is a simplicial complex as above, one considers the spaces
\begin{equation*}
  Z_q^k (X) = \ker d_k \cap
  \ell^q X^k \quad \mathrm{and} \quad B_{p,q}^k(X) =
  d_{k-1}(\ell^p X^k) \cap \ell^q X^k \,.
\end{equation*}
Then the $\ell^{p,q}$-cohomology of $X$ is defined by
\begin{displaymath}
  H_{p,q}^k (X) = Z_q^k (X) /
  B_{p,q}^k(X)\,. 
\end{displaymath}
Its reduced part is the Banach space
\begin{displaymath}
  \overline H_{p,q}^k (X) =
  Z_q^k (X) / \overline B_{p,q}^k(X) \,,
\end{displaymath} 
while its torsion part
\begin{displaymath}
  T_{p,q}^k(X) = \overline B_{p,q}^k(X)/
  B_{p,q}^k(X) \,
\end{displaymath}
is not a Banach space. These spaces fit into the exact sequence
\begin{displaymath}
  0 \rightarrow T_{p,q}^k(X) \rightarrow H_{p,q}^k
  (X)  \rightarrow \overline H_{p,q}^k (X)
  \rightarrow 0 \,.
\end{displaymath}
It is straightforward to check as above that, for $p\geq q$, these
spaces satisfy the same homotopical invariance property as in
Proposition~\ref{prop:invariance-torsion}.

\begin{prop}
  \label{prop:invariance-cohomology}
  Let $X$ and $Y$ be uniformly locally finite simplicial
  complexes. Suppose that they are boundedly homotopic in $\ell^p$ and
  $\ell^q$ norms for $p\geq q$. Then the maps $f_k: \ell^* X^k
  \rightarrow \ell^* Y^k$ and $g_k : \ell^* Y^k \rightarrow \ell^*
  X^k$ induce reciprocal isomorphisms between the $\ell^{p,q}$
  cohomologies of $X$ and $Y$, as well as their reduced and torsion
  components.
\end{prop}

In this setting, the vanishing of the $\ell^{p,2}$-torsion
$T_{p,2}^{k+ 1}(X)$ is equivalent to the closeness of
$B^{k+1}_{p,2}(X) = d_k(\ell^p X^k) \cap \ell^2 X^{k+1}$ in $\ell^2
X^{k+1}$, i.e to the inclusion
\begin{displaymath}
  \overline{d_k(\ell^p X^k) \cap \ell^2 X^{k+1}}^{\ell^2} \subset
  d_k(\ell^p X^k) \cap \ell^2 X^{k+1}  \,.
\end{displaymath}
This implies the weaker inclusion \eqref{eq:9}, but is stronger in
general unless the following holds
\begin{equation}
  \label{eq:26}
  d_k(\ell^p X^k) \cap \ell^2 X^{k+1} \subset \overline{d_k(\ell^2
    X^k)}^{\ell^2} .
\end{equation}
Now by Hodge decomposition in $\ell^2 X^{k+1}$, one has always
\begin{displaymath}
  d_k(\ell^p X^k) \cap \ell^2 X^{k+1} \subset \ker d_{k+1} \cap
  \ell^2 X^{k+1} = \overline{H}_{2}^{k+1}(X) \oplus^\bot \overline{d_k(\ell^2
    X^k)}^{\ell^2} .
\end{displaymath}
Hence \eqref{eq:26} holds if the reduced $\ell^2$-cohomology
$\overline{H}_{2}^{k+1}(X)$ vanishes, proving in that case the
equivalence of \eqref{eq:9} to the vanishing of the
$\ell^{p,2}$-torsion, and even to the identity
\begin{equation}
  \label{eq:27}
  B_{p,2}^{k+1} : = d_k (\ell^p X^k) \cap \ell^2 X^{k+1} = \overline{d_k(\ell^2
    X^k)}^{\ell^2} , 
\end{equation}
which is clearly closed in $\ell^2$.
\begin{cor}
  \label{cor:lp2}
  Let $K$ be a finite simplicial space and $\Gamma \rightarrow X
  \rightarrow K$ a covering. Suppose that the spectral distribution
  $F_{\Gamma,k}$ of $A=d_k^* d_k$ on $(\ker d_k)^\bot$ satisfies
  $F_{\Gamma,k}(\lambda) \leq C \lambda^{\alpha/2}$ for some $\alpha >
  2$. Suppose moreover that the reduced $\ell^2$-cohomology
  $\overline{H}_{2}^{k+1}(X) $ vanishes.

  Then \eqref{eq:27} and the vanishing of the $\ell^{p,2}$-torsion
  $T_{p,2}^{k+1}(X)$ hold for $1/p \leq 1/2 - 1/\alpha$.
\end{cor}

For instance, by \cite{Cheeger-Gromov}, infinite amenable groups have
vanishing reduced $\ell^2$-cohomology in all degrees.

\bibliographystyle{abbrv}


\def\cprime{$'$}

-----------------------------------------------------

\end{document}